\documentclass[12pt,a4paper]{amsart}
\usepackage[top=35mm, bottom=35mm, left=30mm, right=30mm]{geometry}
\usepackage{mathptmx}
\usepackage{mathrsfs}
\usepackage{amscd}
\usepackage{verbatim}
\usepackage{url}
\usepackage[all]{xy}
\usepackage{color}
\usepackage[colorlinks=true,citecolor=blue]{hyperref}
\usepackage{tikz}
\usetikzlibrary{arrows}
\usetikzlibrary{patterns}
\usetikzlibrary{lindenmayersystems,backgrounds}

\theoremstyle{plain}

\newtheorem{thm}{Theorem}[section]
\newtheorem{lem}[thm]{Lemma}
\newtheorem{prop}[thm]{Proposition}
\newtheorem{cor}[thm]{Corollary}
\newtheorem{example}[thm]{Example}

\theoremstyle{definition}

\newcommand{\mZ}{\mathbb Z}

\DeclareMathOperator{\diam}{diam}

\newcommand{\cU}{\mathcal{U}}
\newcommand{\cV}{\mathcal{V}}

\newcommand{\mesh}{{\rm mesh}}

\begin{document}
\title[Distality and expansivity for group actions]{The non-coexistence of distality and expansivity for group actions on infinite compacta}

\author[B.~Liang]{Bingbing Liang}
\address{Soochow University, Suzhou, Jiangsu 215006, China}
\email{bbliang@suda.edu.cn}

\author[E.~Shi]{Enhui Shi}
\address{Soochow University, Suzhou, Jiangsu 215006, China}
\email{ehshi@suda.edu.cn}

\author[Z.~Xie]{Zhiwen Xie}
\address{Soochow University, Suzhou, Jiangsu 215006, China}
\email{20204007002@stu.suda.edu.cn}

\author[H.~Xu]{Hui Xu}
\address{CAS Wu Wen-Tsun Key Laboratory of Mathematics, University of Science and
Technology of China, Hefei, Anhui 230026, China}
\email{huixu2734@ustc.edu.cn}


\date{September, 2021}

\begin{abstract}
Let $X$ be a compact metric space and  $G$  a finitely generated group. Suppose $\phi:G\rightarrow {\rm Homeo}(X)$ is a continuous action.
We show that if $\phi$ is both distal and expansive, then $X$ must be finite. A counterexample is constructed to show the necessity of finite generation
condition on $G$. This is also a supplement to a result due to Auslander-Glasner-Weiss which says that every distal  action by a finitely generated
group on a zero-dimensional compactum is equicontinuous.

\end{abstract}

\maketitle

\section{Introduction}

Expansivity and distality are two important notions in the theory of dynamical system. The notion of distality was introduced by Hilbert for a better understanding of equicontinuity \cite{E58}. Later,  Furstenberg  described completely the relations between distality and equicontinuity for minimal group actions, which is nowadays known as the Furstenberg's
structure theorem \cite{Fur63}. Readers may refer to \cite{S20} and  \cite{Sh20} for some recent studies around the topological and algebraic aspects of minimal distal systems. The notion of expansivity is closely related to the theory of structure stability in differential dynamical systems, which is also a kind of chaotic property
possessed by many natural systems such as the Anosov systems and symbolic systems. One may consult \cite{A89} for a systematic introduction to this property.
\medskip

Now let us first recall the definitions of these two notions. By a {\it compactum} we mean a compact metric space.  Let $(X, G)$ be a continuous action of a group $G$ on a compactum $X$ with the metric $d$. If $x, y\in X$ are such that $\inf_{g\in G}d(gx, gy)=0$, then $x$ and $y$ are said to be {\it proximal}; if $x=y$, then the proximal
pair $\{x, y\}$ are said to be {\it trivial}. If $(X, G)$ contains no nontrivial proximal pair, then it is said to be {\it distal}; if there is a constant $c>0$ such that
for every $x\not=y\in X$, $\sup_{g\in G}d(x, y)>c$, then $(X, G)$ is said to be {\it expansive}.  Such a constant $c$ is called an {\it expansive constant} of $(X, G)$. Intuitively, these two notions describe two kinds of
dynamical phenomena which lie on separately two opposite ends of complexity. This is indeed what happens for $G=\mZ$ as the following well-known
result indicates: no infinite compact metric space admits an expansive and distal $\mZ$-action  \cite[Propsition 2.7.1]{BS02}.
\medskip

The aim of the paper is to study the coexistence of distality and expansivity for general group actions. The study of expansive systems for general group actions attracts researchers' constant attention.
A classical result due to Ma\~n\'e shows that if a compactum $X$ admits an
expansive $\mathbb Z$-action, then $X$ is of finite dimension \cite{M79}.
Unfortunately, this result fails for $\mZ^2$-actions. Shi and Zhou constructed an expansive $\mathbb Z^2$-action on an infinite
dimensional compactum \cite{SZ05}. However, Meyerovitch and Tsukamoto shows that for an expansive $\mZ^k$-action the mean dimension of $\mZ^{k-1}$-subgroup action is finite, which provides a proper generalization of Ma\~n\'e's result \cite{MT19}. For Suslinian continua it is shown that there is no expansive action by groups of subexponential growth \cite{BSXX21}. This extends a result due to Kato for $\mathbb Z$-actions \cite{K90}. There is also a cumulative development on the expansive actions by continuous automorphisms on compact metrizable abelian groups appealing tools from commutative algebras, operator algebras, etc. \cite{LS99, CL15, Mey19, BGL20}.

\medskip

The main result of the paper is:

\begin{thm} \label{main result}
Let  $G$ be a finitely generated group and $X$ a compact metric space. If a continuous action $(X, G)$ is both distal and expansive, then $X$
is finite.
\end{thm}

Relying on the heavy machine of the Furstenberg's structure theorem, in Section 3, we finish the proof of this theorem.
\medskip

In Section 4 we give a counterexample to show that the finitely generating condition is indispensable. Specifically, we construct a
distal and expansive minimal action on the Cantor set by a non-finitely generated countable group. This is also a supplement
to a result due to Auslander-Glasner-Weiss  \cite[Corollary 1.9]{AGW07}, which says that if a zero-dimensional compactum $X$ admits a distal minimal action by
a finitely generated group, then the action is equicontinuous. Our example indicates that there does exist a non-equicontinous
distal action on a zero-dimensional compactum by a non-finitely-generated group.

\section{Prelimilaries}

Let $G$ be a countable discrete group with the identity $e_G$ and $X$  a compactum with metric $d$.  By a {\it continuous action} of  $G$ on $X$ or simply a {\it dynamical system}, denoted by  $(X, G)$, we mean a
group homomorphism $\phi \colon G\rightarrow {\rm Homeo}(X)$. For brevity, we usually use $sx$ to denote $\phi(s)(x)$ for $s\in G$ and $x\in X$. For $x\in X$, the set $Gx:=\{gx: g\in G\}$
is called the {\it orbit} of $x$ under the action of $G$. A continuous surjection $\pi \colon (X, G) \to (Y, G)$ between two actions is called a {\it factor map} or an
{\it extension} if  $\pi(gx)=g\pi(x)$ for every $g \in G$ and $x \in X$.
 A dynamical system $(X, G)$ is called {\it minimal} if $X$ has no proper nonempty $G$-invariant closed subsets; in other words, every orbit of $(G, X)$ is dense in $X$.  We say $(X, G)$  is {\it equicontinuous} if for any $\varepsilon > 0$ there exists $\delta > 0$ such that
$\sup_{g \in G} d(gx, gy) < \varepsilon$ for all $x, y\in X$ with $d(x, y) < \delta$. A factor map $\pi \colon (X, G) \to (Y, G)$ is called  {\it equicontinuous extension} if for any $\varepsilon > 0$ there exists $\delta > 0$ such that for all $x,x' \in X$ with $d(x, x') < \delta$ and $\pi(x)=\pi(x')$ we have  $\sup_{g \in G} d(gx, gx') < \varepsilon.$
Clearly equicontinuous systems are distal and cannot be expansive unless the system is finite.
\medskip

We collect some results which will be used in the sequel.
\medskip

Let $(X, d)$ be a compact metric space. Denote by $2^X$ the collection of nonempty closed subsets of $X$, which is called the {\it hyperspace over $X$}.  Under the Hausdorff metric $d_H$ induced from $d$ the hyperspace $2^X$ is a compact metric space \cite[Theorems 4.13]{N92}. Let $f \colon Z \to Z'$ be a continuous map between compact metric spaces. Note that the induced map $f^{-1} \colon Z' \to 2^{Z}$ sending $u$ to $f^{-1}(u)$ is always upper semi-continuous;  $f^{-1}$ is lower semi-continuous if and only if $f$ is an open map. In summary we have $f^{-1}$ continuous if and only if $f$ is an open map \cite[Appendix A.7]{Vri93}.

\begin{lem}\label{open map}\cite[Theorem 8.1]{Fur63}
 Let $\pi: (X,G) \rightarrow (Y,G)$ be a factor map of minimal distal systems. Then $\pi$ is open. In particular,  the map $\pi^{-1}: Y\rightarrow 2^X$ is continuous.
\end{lem}

\medskip

Let $\xi$ be an ordinal. A {\it projective system} is a collection   of dynamical systems $\{(X_\lambda, G)\}_{\lambda\leq\xi}$ together with a collection of factor maps
 $\{\pi_{\alpha,\beta}\colon X_\alpha\rightarrow X_\beta\}_{\beta<\alpha\leq\xi}$ satisfying $\pi_{\alpha,\gamma}=\pi_{\beta,\gamma}\circ \pi_{\alpha,\beta}$ for any
 $\gamma<\beta<\alpha$; the {\it projective limit} of this projective system, denoted by $\underset{\longleftarrow}{\lim}(X_{\lambda}, G)_{\lambda\leq\xi}$, consists of the closed subset
 $X:=\{(x_\lambda)\in \prod_{\lambda\leq\xi}X_{\lambda}: \pi_{\alpha,\beta}(x_\alpha)=x_\beta, \forall \ \alpha<\beta\leq\xi\}$ of $\prod_{\lambda\leq\xi}X_{\lambda}$
 and the diagonal action of $G$ on $X$ defined by $g(x_\lambda)=(gx_\lambda)$.
\medskip

Furstenberg established the following structure theorem of minimal distal systems \cite[Theorem 2.3]{Fur63}.
\begin{thm} \label{distal structure}
For any minimal distal system $(X,G)$, there exists a countable ordinal $\eta$ and to each ordinal $\alpha\leq \eta$ there is an associated minimal factor system $(X_{\alpha},G)$ such that
\begin{itemize}
\item [(i)] $(X_1, G)$ is the maximal equicontinuous factor of $(X,G)$ where $1$ is the least ordinal, and $(X_{\eta}, G)=(X,G)$;
\item [(ii)] for any $\beta<\alpha\leq \eta$, there is a nontrivial extension $\pi_{\alpha,\beta}\colon X_{\alpha}\rightarrow X_{\beta}$ such that for any $\gamma<\beta$, we always have
$\pi_{\alpha,\gamma}=\pi_{\beta,\gamma}\circ \pi_{\alpha,\beta}$;
\item [(iii)] for any successor ordinal $\alpha\leq \eta$ , $\pi_{\alpha,\alpha-1}\colon X_{\alpha}\rightarrow X_{\alpha-1}$ is an equicontinuous extension;
\item [(iv)] for any limit ordinal $\xi$,  $(X_{\xi},G)=\underset{\longleftarrow}{\lim}(X_{\alpha}, G)_{\alpha< \xi}$.
\end{itemize}

\end{thm}

We call the ordinal $\eta$ appeared in the above theorem the {\it height} of $(X, G)$. In general, a dynamical system can not be partitioned as the disjoint union of minimal subsets. For distal systems, we do  have such a decomposition.
\begin{lem}\label{min sets}\cite[Theorem 3.2]{Fur63}
Any distal system is a disjoint union of minimal subsets.
\end{lem}

\begin{thm}\label{AGW07}\cite[Corollary 1.9]{AGW07}
Let $G$ be a finitely generated group and  $(X,G)$ a distal system. If $X$ is zero-dimensional, then $(X,G)$ is equicontinuous.
\end{thm}

\section{Proof the main theorem}

To prove the main theorem, we first make some preparations.

\begin{lem} \label{cover map}
Let $\pi \colon (X,G)\rightarrow (Y,G)$ be an equicontinuous extension between minimal distal systems. If $(X,G)$ is expansive, then there is $N\in\mathbb{N}$ such that $|\pi^{-1}(y)| = N$ for any $y\in Y$
and $\pi$ is a covering map.
\end{lem}

\begin{proof}
Let $c$ be an expansive constant for $(X,G)$.  Since the extension $\pi$ is equicontinuous, we have $\inf\{d(x,x'):\pi(x)=\pi(x'),x\neq x'\}>0$. Thus for any $y\in Y$, by the compactness of $\pi^{-1}(y)$,
the cardinality $|\pi^{-1}(y)| <\infty$.

Since $(Y,G)$ is minimal, for any $y_1, y_2\in Y$, there exists a sequence $g_n$'s of $G$ approaching to $y_2$.
By Lemma \ref{open map}, $\pi^{-1}(g_ny_1)$ approaches to $\pi^{-1}(y_2)$ in $2^X$ under the Hausdorff metric. Thus $|\pi^{-1}(y_1)|\geq |\pi^{-1}(y_2)|$.
By symmetry, we have $|\pi^{-1}(y_1)|\leq |\pi^{-1}(y_2)|$ and hence $|\pi^{-1}(y_1)|= |\pi^{-1}(y_2)|$. This implies that, for every
$x\in X$, there is an open neighborhood $U$ of $x$ such that the restriction $\pi|_U: U\rightarrow Y$ is injective. Since $\pi$ is open
by Lemma \ref{open map}, $\pi|_U$ is a homeomorphism from $U$ to $\pi(U)$. Thus $\pi$ is a covering map.
\end{proof}

Recall that a covering map $\pi \colon X\rightarrow Y$ between two compact metric spaces is said to be of {\it finite-fold} if $\sup_{y \in Y} |\pi^{-1}(y)|< \infty$.
By compactness, we have

\begin{lem}\label{lower bound}
If $\pi: X\rightarrow Y$ is a finite-fold covering between two compacta, then $\inf\{d(x,x'): x\not=x'\in X,\pi(x)=\pi(x')\} > 0$.
\end{lem}

Let $X$ be a metric space with metric $d$ and $A\subset X$. If there is some $c>0$ such that $d(x, y)>c$ for any $x\not=y\in A$, then $A$ is called
{\it $(c, d)$-separated}. We use the symbol ${\rm diam}(A, d)$ to denote the diameter of $A$ with respect to the metric $d$. If
$\mathcal U$ is a family of subsets of $X$, then ${\rm mesh}({\mathcal U}, d):=\sup\{{\rm diam}(A, d): A\in \mathcal U\}$. If $G$ is a group
with a finite generating set $S$, then the {\it word length} of $g\in G$ is the minimal $n$ such that $g$ can be expressed as $g=s_1s_2\cdots s_n$ where
$s_i\in S$ or $s_i^{-1}\in S$.
\medskip

The following lemma is a generalization of \cite[Theorem 2.23]{A89}. The finitely generating condition  is indispensable as the example will illustrate in Section 4.

\begin{lem}\label{cover}
Let $G$ be a finitely generated group and  $\pi: (X,G)\rightarrow (Y,G)$  a factor map. If $\pi$ is a finite-fold covering and $(X,G)$ is expansive, then $(Y, G)$ is  expansive.
\end{lem}
\begin{proof}
Let $d$ be the metric on $X$. Fix a  finite generating subset $S$ of $G$ with $S=S^{-1}$. Then $d$ induces a compatible metric $d'$ defined by
$$d'(x,x'):=\max_{s \in S \cup \{e_G\}} d(sx, sx').$$
In light of Lemma \ref{lower bound}, we may fix an expansive constant $c$ of $(X, G)$ such that $d(x, x') > c$ for every distinct $x, x' \in X$ with $\pi(x)=\pi(x')$. In other words, every fiber of $\pi$ is $(c, d)$-separated.

By the definition of finite-fold covering, there exists a positive integer $k$ such that $|\pi^{-1}(y)| \leq k$ for every $y \in Y$. Then for every $y\in Y$ there is an open  neighborhood $V_y$ of $y$ such that
\begin{enumerate}
    \item $\pi^{-1}(V_y)$ can be partitioned as a disjoint union of some open sets $\{ U_{y,x}: x \in \pi^{-1}(y) \}$;
    \item $\pi$ maps each $U_{y,x}$ homeomorphically onto $V_y$.
\end{enumerate}
WLOG, we may suppose that $\diam (U_{y, x}, d') < c/2$. By the compactness of $Y$, we then obtain a finite open cover $\cV=\{V_{y_1},\cdots, V_{y_n}\}$ of $Y$ and a finite open cover $\cU=\{U_{y_i, x}:  x \in \pi^{-1}(y_i), i=1, \cdots, n \}$ of $X$.

Let $\beta$ be a Lebesgue number for $\mathcal{V}$. Assume that $(Y, G)$ is not expansive. Then there exist distinct $p, p' \in Y$ such that $d(up, up') < \beta$ for all $u \in G$. In particular, $p, p' \in V$ for some $V=V_{y_i} \in \cV$. From our choice of $\cU$ for every $x \in \pi^{-1}(y_i)$ there exists unique $q, q' \in U_{y_i, x}$ such that $\pi(q)=p$ and $\pi(q')=p'$. We shall fix such a $x$ and $q, q'$.

Since $(X, G)$ is expansive and $\mesh(\cU, d') < c/2$, there exists $g \in G\setminus \{e_G\}$ of the smallest word length such that $gq, gq'$ fail to be in a common element of $\cU$. Write $g=sh$ for some $s \in S$ and $h \in G$ with smaller word length. It forces that $hq, hq'$ belong to some element $U$ of $\cU$ and hence $gq, gq' \in sU$. Again since $\mesh(\cU, d') < c/2$, we obtain
$$d(gq, gq')  < c/2.  \eqno(*)$$

On the other hand, since $d(gp, gp') < \beta$, we have $gp, gp' \in V'$ for some $V'=V_{y_j}$ of $\cV$. Then there exists a unique $x' \in \pi^{-1}(y_j)$ such that $gq \in U_{y_j, x'}$. Furthermore, there exists a unique $\overline{q} \in U_{y_j, x'}$ such that $\pi(\overline{q})=gp'$ and hence $d(gq, \overline{q}) < c/2$. Together with $(*)$, we obtain
$$d(gq', \overline{q}) \leq d(gq, gq')  + d(gq, \overline{q}) <c.$$
By our choice of $g$, $gq'\neq \overline{q}$. This contradicts to our assumption that every fiber of $\pi$ is $(c, d)$-separated.

\end{proof}

Combining Lemmas \ref{cover map} with \ref{cover}, we obtain
\begin{prop}\label{iso extension}
Let $\pi: (X,G)\rightarrow (Y,G)$ be an equicontinuous extension between minimal distal systems.   If $(X,G)$ is  expansive, then so is $(Y,G)$.
\end{prop}

We remark that in general the expansivity does not pass to the factor systems (even for $\mathbb{Z}$-actions).
The following proposition shows that strict projective limit of dynamical systems ruin expansivity.
\begin{lem}\label{inverse limit}
Let $\eta$ be a countable limit ordinal and $\{(X_{\alpha}, G)\}_{\alpha< \eta}$ be a projective system with factor maps $\{\pi_{\alpha, \beta}\}_{\beta<\alpha<\eta}$.
Suppose $(X,G)=\underset{\longleftarrow}{\lim}(X_{\alpha}, G)_{\alpha<\eta}$ and every $\pi_{\alpha, \beta}$ is nontrivial. Then $(X, G)$ is not expansive.
\end{lem}

\begin{proof} Fix a compatible metric $d$ on $X$. Assume to the contrary that $(X, G)$ is expansive with an expansive constant $c$. Pick $\alpha < \eta$ such that $d(x, x') < c$ for any $x, x' \in X$ with $\pi_{\eta, \alpha}(x)=\pi_{\eta, \alpha}(x')$.   Since every $\pi_{\alpha, \beta}$ is nontrivial, there exists distinct $u, u' \in X$ such that $\pi_{\eta, \alpha}(u)=\pi_{\eta, \alpha}(u')$.  It forces that $d(gu, gu') < c$. This contradicts to the assumption of expansiveness.
\end{proof}

\begin{prop}\label{same period}
Let  $G$ be a finitely generated group and $(X,G)$ be an infinite dynamical system. If $(X, G)$ is {\it pointwise periodic}, i.e. each orbit is finite,  then $(X,G)$ is not expansive.
\end{prop}
\begin{proof}

Set $E_0=\emptyset$ and  $E_k=\{x: |Gx|\leq k\}$ for each positive integer $k$. If every $E_k$ is finite,  then $X$ is countable and hence $0$-dimensional. Note that pointwise periodic systems are distal. Applying Theorem \ref{AGW07}, we have $(X, G)$ is equicontinuous. Since $X$ is infinite, it forces that $(X, G)$ is not expansive.  Thus we may assume $E_k$ is infinite for some $k$ and $p$ is the smallest such that $E_p$ is infinite. In particular, $E_{p-1}$ is finite.

If all cluster points of $E_p$ sit inside $E_{p-1}$, we can choose some convergent pairwise distinct sequence $x_n$'s in $E_p$ approaching to a cluster point of $E_p$. Then the union of $G_{x_n}$'s with $E_{p-1}$ makes an infinite subsystem of zero dimension. By Theorem \ref{AGW07}, this subsystem is equicontinuous, which forces that $(X, G)$ is not expansive.
 Thus we may assume that there is a cluster point $z$ of $E_p$ still sitting inside $E_p$.

Let $x_n$'s be a sequence in $E_p \setminus \{z\}$ converging to $z$. Assume to the contrary that $(X, G)$ is expansive. Pick an expansive constant $c$ of $(X, G)$ such that $d(x, x') > 2c$ for all distinct element $x, x'$ of $Gz$. Fix a finite generator set $S$ of $G$ with $S=S^{-1}$.
Take $0 < \delta < c$ such that
$d(sx, sx') < c$
for every $s \in S$ and $x, x' \in X$ with $d(x, x') < \delta$.
Write $Gz=\{z, g_1z, \cdots, g_{p-1}z\}$ and $g_0=e_G$. Then as $n$ large enough, we have for every $i=0,1,\cdots, p-1$,
$d(g_ix_n, g_iz) < \delta$.
We shall fix such a $x_n$.

By expansivity, there exists $g \in G\setminus \{e_G\}$ such that $d(gz_n, gz) \geq \delta$. Assume that $g$ is such an element with the smallest word length. Then $g=sh$ for some $s \in S$ and $h \in G$ with smaller length.  It follows that $d(hz_n, hz) < \delta$ and hence $d(gz_n, gz) < c$. On the other hand, $gz_n=g_iz_n$ for some $0 \leq i \leq p-1$. So
$$d(gz, g_iz) < d(gz, gz_n)+ d(gz_n, g_iz) < c +\delta < 2c.$$
By our choice of  $c$, it forces that $gz = g_iz$. Therefore,
$$d(gz_n, gz)=d(g_iz_n, g_iz) < \delta,$$
which contradicts to the choice of $g$.

\end{proof}

Recall that a {\it general subshift of symbolic systems} is a subsystem of $(A^G, G)$ carrying the shift action of $G$ for some finite set $A$ (also called Bebutov system).
Since every subshift of symbolic systems is expansive, it follows that
\begin{cor}
Let $G$ be a finitely generated group and $(X, G)$ is a subshift of symbolic systems. If $(X, G)$ is pointwise periodic, then it is finite.
\end{cor}

The following example indicates that the finitely generating condition on $G$ in Proposition \ref{same period} is necessary.

\begin{example}
Denote by $A$ the subspace $\{0,1, 1/2, 1/3, \cdots\}$ of $\mathbb{R}$ under the   Euclidean topology and $X$ the product space $A\times \mZ/2\mZ$. For every $n \in \mathbb{N}$ define a homeomorphism $\sigma_n \colon X \to X$ by sending $(x, y)$ to $(x, y+\overline{1})$ if $x=1/n$ and $(x,y)$ elsewhere. Then $\{\sigma_n\}_{n\geq 1}$ generate a countable group $G$  acting on $X$. Clearly $G$ is not finitely generated and $(X, G)$ is pointwise periodic and expansive.
\end{example}

Now we are ready to prove the main theorem of the paper.

\begin{proof} [Proof of Theorem \ref{main result}]

By Lemma \ref{min sets}, $(X,G)$ is the disjoint union of minimal subsets.
In light of Proposition \ref{same period}, without loss of generality, we may assume that $(X, G)$ is minimal.

 As in  Theorem \ref{distal structure} we can write $(X, G)$ as a projective limit of  extensions $\{\pi_{\alpha, \beta} \colon X_\alpha \to X_\beta\}_{\beta < \alpha \leq \eta}$ with the height $\eta$. We do induction on the height $\eta$ of $(X, G)$.
For $\eta=1$, it means that  $(X, G)$ is equicontinuous and the conclusion obviously holds. Suppose that the conclusion holds for any minimal distal system with height less than $\eta$. In light of Lemma \ref{inverse limit}, $\eta$ can only be a successor ordinal. So $(X, G)$ is the isometric extension of $(X_{\eta -1}, G)$. By
Proposition \ref{iso extension}, $(X_{\eta-1}, G)$
is expansive. Applying the induction assumption, $X_{\eta -1}$ is finite. From  Lemma \ref{cover map}, $(X, G)$ is a finite-fold covering of
a finite system and hence is finite again.
\medskip

\end{proof}

As a consequence, we have
\begin{cor}
Let $G$ be a finitely generated group and $(X, G)$ an infinite and expansive system. Then there exists a nontrivial proximal pair in $(X, G)$.
\end{cor}

\section{A counterexample}
In this section, we give a counterexample to show that the finitely generating  condition in Theorem \ref{main result} is indispensable.
\medskip

Recall that  the {\it dyadic odometer} is the additive group $\Sigma:=(\{0,1\}^{\mathbb{Z}_{+}},+)$ associated with the homeomorphism $T_0$ sending $x$ to $x+(1,0,0,\cdots)$.   For $\xi=(\xi_1,\xi_2,\cdots), \eta=(\eta_1,\eta_2,\cdots) \in \Sigma$, define
$ d_1(\xi,\eta)=\frac{1}{n}$ for $n=\min\{ i\in\mathbb{Z}_{+}: \xi_i\neq \eta_i\}$.
Then $d_1$ is a metric on $\Sigma$ and is $T_0$-invariant.

Let $X=\Sigma\times \mathbb{Z}/2\mathbb{Z}$ be endowed a metric via
\[ d\Big( (\xi, a), (\eta, b) \Big)=d_1(\xi,\eta)+\theta(a,b),\]
where $\theta$ is the discrete metric on $\mZ/2\mZ$.
Then $(X,d)$ is an infinite compactum of zero dimension. It is clear that $g$ induces an isometry $T \colon X \to X$ by sending $(\xi, a)$ to $(T_0(\xi), a)$.
For every $n\in\mathbb{Z}_{+}$, set $\Sigma_n= \{0,1\}^{\{1,2,\cdots,n\}}$ and $\pi_n \colon \Sigma \to \Sigma_n$ the projection map.
Each $\omega\in\Sigma_n$ induces a permutation $ f_{\omega}: X\rightarrow X$ by
\[f_{\omega}(\xi, a)=\begin{cases}(\xi, a+\overline{1}), & {\rm{if}} ~~\pi_n(\xi)=\omega;\\ (\xi, a),&  {\rm{else.}} \end{cases}\]
Let $G$ be the group generated by  $S=\{T\} \cup\bigcup_{n\in\mathbb{Z}_{+} }\{ f_{\omega}: \omega\in \Sigma_n \}$. It is clear that $G$ is non-finitely-generated  and $G$ acts on $X$ by homeomorphisms.

Since $(X, G)$  is an equicontiuous extension of the distal action $(\Sigma, G)$, it is distal. We check that $1/2$ is an expansive constant for $(X,G)$.
For distinct $(\xi, a) , (\eta, b)\in X$, if  $a\neq b$, we have $ d((\xi, a) , (\eta, b))\geq 1>1/2$. Thus we may assume $a=b$. Then  $\xi\neq\eta$ and hence $\xi_n\neq\eta_n$ for some $n\in\mathbb{Z}_{+}$.  Put $\omega=\pi_n(\xi)$. We have
\[d(f_{\omega}(\xi, a) ,f_{\omega} (\eta, b))=d((\xi, a+\overline{1}), (\eta, a))\geq \theta(a+\overline{1}, a)=1>1/2.\]
It is also easy to see that $(X, G)$ is minimal.

In summary, we obtain the following proposition, which  is a complement to Theorem \ref{AGW07}.
\begin{prop}
There exists a zero dimensional space $X$ and a minimal, distal but not equicontinuous action on $X$ by a non-finitely-generated group $G$.
\end{prop}

\end{document}